\documentclass{amsart}

\usepackage{fancyheadings}
\usepackage{amsfonts}
\usepackage{graphicx}
\usepackage{amscd,amsmath}
\usepackage{fancyhdr}
\usepackage{amssymb}
\usepackage{setspace}
 \newtheorem{thm}{Theorem}[section]
\newtheorem{eg}[thm]{Example}

\newtheorem{cor}[thm]{Corollary}

\theoremstyle{remark}

\begin{document}
\large
\title[Extended Generalised Flett's MVT]{Extended Generalised Flett's Mean Value
Theorem}
\author [Rupali Pandey \& Sahdeo Padhye ]{Rupali Pandey \& Sahadeo Padhye}
\address{Department of Mathematics, Motilal Nehru National Institute of Technology Allahabad, Allahabad
(UP),India.}

\email{pandeyrupali1992@gmail.com, sahadeo@mnnit.ac.in}


%\thanks{This work is supported under CSIR (JRF) scheme, India (2002).}
\maketitle

Various forms of Mean Value Theorems are available in the
literature. If we use Flett's Mean Value Theorem in Extended
Generalized Mean Value Theorem then what would the new theorem look
like. A sincere effort is done to develop this theorem. This theorem
is named as Extended Generalised Flett's Mean Value Theorem.

\vspace{.25cm}

%\noindent 2000 Mathematics Subject Classification- 26A06.

%\noindent \textbf{Key Words}: Function of One Variable, Mean Value
%Theorem, Extended Generalized Mean Value Theorem, Flett's Mean Value
%Theorem

\section {introduction}

Lagrange's Mean Value Theorem (LMVT) often called Mean value Theorem
(MVT) is one of the most important result in mathematical analysis.
It is important tool used in differential and integral calculus. For
example, it is useful in proving Fundamental Theorem of Calculus.
Every mathematics student knows the Lagrange's mean value theorem
which has appeared in Lagrange's book Theorie des functions
analytiques in 1797 as an extension of Rolle's result from 1691.
Mean Value Theorem says something about the slope of a function on
closed interval based on the values of the function at the two
endpoints of the interval. It relates \textit{local} behavior of the
function to it's \textit{global} behavior. This theorem turns out to
be the key to many other theorems about the graph of functions and
their behavior. The MVT can be stated as below.

\noindent \textbf{Lagrange's Mean Value Theorem :}
Suppose $ f : [a, b]\rightarrow R $ is a function satisfying two conditions: \\
\noindent 1. $ f $ is continuous on closed interval $ [a,b] $ \\
2. $ f $ is differentiable on open interval $ (a,b) $ \\
Then there  $ \exists $ a number $ c $ in $ (a,b) $ such that \\
$ f'(c) = \dfrac{f(b)-f(a)}{b-a} $

Many extension, generalization and forms of MVT were proposed by
many mathematicians from all over the worlds. In this article we
will discuss Flett's Mean Value Theorem (FMVT) ~\cite{Flett58} and
Extended Generalized Mean Value Theorem (EGMVT)~\cite{Phi14}. Then
we will state a new form of MVT based on FMVT and EGMVT and we
called it Extended Generalized Fleet's Mean Value Theorem (EGFMVT).
We also provide an example for the support of our proposed theorem.

\section{Flett's Mean Value Theorem}

In 1958, Thomas Muirhead Flett ~\cite{Flett58} gave the idea of a
Mean Value Theorem using different conditions as used in Mean Value
Theorem, we called it Flett's mean value theorem (FMVT). The FMVT
was based on some observations: 1.To notice the changes if in
Rolle's Theorem the hypothesis $ f(a)=f(b) $ refers to higher order
derivatives? 2.To see is there any analogy with LMVT? 3.To see which
geometrical consequences do such result have?  ~\cite{Flett12}.

\noindent \textbf{Notations}: Throughout in this article we will use
following notations.

\noindent 1. $ C(M) \rightarrow $ space of continuous functions. \\
2. $ D^{n}(M) \rightarrow $ $ n $ -times differentiable real-valued function on set $ M \subseteq R $. \\
3. Under continuity of a function on $ [a,b] $ we understand its continuity on $ (a,b) $ and one- sided continuity at end points of interval .\\
4. Similarly differentiability holds under open interval. \\
5. For a function $ f,g $ on an interval $ [a,b] $ the expression of the form.\\

  $ \dfrac{f^{n}(b)-f^{n}(a)}{g^{n}(b)-g^{n}(a)},  n \in \mathbb{N} \cup \lbrace 0 \rbrace $ will be denoted by symbol $ {_a^b}K(f^{n},g^{n}) $. If
denominator is equal to $ b-a $ we will only write $ {_a^b}K(f^{n}) $.\\
So the Lagrange Theorem in the introduced notation has the form: \\
if $ f \in C[a,b] \cap D(a,b) $ then $ \exists $ $ \eta \in (a,b) $
s.t. $ f'(\eta)={_a^b}K(f) $, where we use usual convection $
f^{(0)}=f $.

\noindent Basically, Flett focus on the theme of Rolle's Theorem
where condition $ f(a)=f(b) $ is replaced by $ f'(a)=f'(b) $. It is
a Lagrange's Type Mean Value Theorem with Rolle's Type Condition .

\begin{thm}

\textbf{Fletts Mean Value Theorem-I ~\cite{Flett58}}If $ g \in
C[a,b] $, then from the integral mean value theorem $ \exists $  $
\eta \in (a,b) $ s.t.
\begin{equation}
g(\eta)= \frac{1}{b-a} \int_{a}^{b} g(t) dt
\end{equation}
\end{thm}

\begin{thm}  \textbf{Fletts Mean Value Theorem-II ~\cite{Flett58}}
If $ f \in C[A,b] \cap  D(a,b) $ and $ f'(a)=f'(b) $, then $ \exists
$ $ \eta \in (a,b) $ s.t. $ f'(\eta)={_a^\eta}K(f) $ .
\end{thm}
We refer the reader ~\cite{Flett58, Flett12} for the proof of the
above two theorems.

\begin{eg}
~\cite{Flett12} In which point of curve $ y = x^{3} $ the tangent
passes through the point $ X=[-2,2] $?
\end{eg}
\noindent \textbf{Solution} - To verify $ X $ lies on curve and $ y
= x^{2} $ differentiable on $ \mathbb{R} $.
Since $ y'= 3x^{2} $ is even function on $ \mathbb{R} $. \\
 $ y'(-2)= 12= y'(2) $ \\
By Flett's Theorem, $ \exists $  $ \eta \in (-2,2) $ s.t.
\begin{eqnarray*}
f'(\eta) & = & {_{-2}^\eta}K(f) \\
\Rightarrow 3 \eta^{2} & = & \frac{\eta^{3}+8}{\eta+2} \\
\Leftrightarrow \eta^{2}+3\eta-4 & = & 0 \\
\Leftrightarrow (\eta+4)(\eta-1) & = & 0
\end{eqnarray*}
Because $ -4 \notin (-2,2) $, we consider only $ \eta=1 $. Then $ y(\eta)=1 $ \\
The desired point is $ T=[-1,1] $.

\section{Extended Generalized Mean Value Theorem}

In 2014, Phillip Mafuta ~\cite{Phi14}, stated and proved a theorem
of a similar flavor to the Generalized Mean Value Theorem for
functions of one variable. For lack of a better term, he called the
theorem ``Extended Generalized Mean Value Theorem (EGMVT)''. In
addition, he applied Rolle's Theorem to prove the EGMVT. Also, he
deduced some corollaries for Mean Value Theorems. In addition, the
EGMVT is verified by an example.

\begin{thm} ~\cite{Phi14}
Let $ n \in \mathbb{N} $ and let $
f(x),g_{1}(x),g_{2}(x),g_{3}(x),......,g_{n}(x) $ be $ n+1 $
continuous functions on a closed bounded interval $ [a,b] $ and
differentiable in an open interval $ (a,b) $ with $ g'_{i}(x) \neq
0, \forall x \in (a,b) $ for $ i=1,2,3,....,n $. Then $ \exists $  $
\xi
\in (a,b) $ : \\

$ f'(\xi)= \dfrac{f(b)-f(a)}{b-a}[\sum_{i=1}^{n}
\dfrac{g'_{i}(\xi)}{g_{i}(b)-g_{i}(a)}] $
\end{thm}

\begin{proof}
Define a function $ F(x)$ by: \noindent \begin{equation} F(x) =
n(f(x)-f(a))-\sum_{i=1}^{n}
\dfrac{f(b)-f(a)}{g_{i}(b)-g_{i}(a)}(g_{i}(x)-g_{i}(a))
\end{equation}
Since $ g'(x) \neq 0 \Rightarrow g_{i}(b)-g_{i}(a) \neq 0 $ , for $ i=1,2,3,....n $\\
Moreover, $ n $ is finite. So $ F(x) $ is well defined on closed interval $ [a,b] $, by algebra of continuous functions.\\
In addition,it follows by algebra of differentiable functions that $ F(x) $ is differentiable function in $ (a,b) $.\\
Hence the condition for Rolle's Theorem are satisfied, so $ \exists
$ $ \xi \in (a,b)$ s.t. $ F'(\xi)=0 $ \noindent \begin{equation}
\Rightarrow nf'(\xi) - \sum_{i=1}^{n}
\dfrac{f(b)-f(a)}{g_{i}(b)-g_{i}(a)}g'_{i}(\xi)=0
\end{equation}
Thus, $ \exists $  $ \xi \in (a,b): f'(\xi)=
\frac{f(b)-f(a)}{n}[\sum_{i=1}^{n}
\dfrac{g'_{i}(\xi)}{g_{i}(b)-g_{i}(a)}] $.
\end{proof}

\begin{eg}
 ~\cite{Phi14} Consider $ f(x) = x+1, g_{1}(x) = x^{2}+4x-4 $ and $ g_{2}(x) =
x^{2}+3x $ on an interval $ [0,3] $.
\end{eg}
\noindent \textbf{Solution} :
1. All functions are continuous on $ [0,3] $.\\
2. All functions are differentiable on $ (0,3) $.\\
with $ g'_{i}(x) \neq 0 $, $ \forall x \in (0,3) $ for $ i=1,2 $ \\
So condition of \textbf{EGMVT} are satisfied hence, $ \exists $ $
\xi \in (0,3) $ :
\begin{equation}
f'(\xi) = \frac{f(3) - f(0)}{2}[\sum_{i=1}^{2}
\dfrac{g'_{i}(\xi)}{g_{i}(3)-g_{i}(0)}]
\end{equation}
$ \Rightarrow 78 \xi = 117 $ \\
so, $ \xi = 1.5 \in (0,3) $ as required.

\begin{cor}
 ~\cite{Phi14} Let $ n \in \mathbb{N} $ and let $ f(x), g_{1}(x), g_{2}(x),
g_{3}(x), ....., g_{n}(x) $ be $ n+1 $ continuous functions on
closed bounded interval $ [a,b] $ and are $ k $ times differentiable
in an open interval $ (a,b) $ with $ g_{i}^{(k)} \neq 0, \forall x
\in (a,b) $ for $ i=1,2,3,..... n $. Then $ \exists $ $ \xi \in
(a,b) $ s.t.
\begin{equation}
f^{k}(\xi) = \frac{f^{k-1}(b) - f^{k-1}(a)}{n}[\sum_{i=1}^{n}
\dfrac{g_{i}^{(k)}(\xi)}{g_{i}^(k-1)(b)-g_{i}^(k-1)(a)}]
\end{equation}
\end{cor}

\begin{proof}
Set $ \phi(x) = f^{k-1}(x) \\
      \psi_{i}(x) = g_{i}^(k-1)(x) $ for $ i=1,2,3,......,n $ \\
$\Rightarrow  \psi'_{i}(x) = g_{i}^(k)(x) \neq 0 $ \\
Moreover both $ \phi (x) $ and $ \psi^{i}(x) $ are continuous on $
[a,b] $ and differentiable on $ (a,b) $. Hence by EGMVT,
$ \exists $ $ \xi \in (a,b) $ s.t. \\
$ \phi'(\xi) = \frac{\phi(b) - \phi(a)}{n} [\sum_{i=1}^{n}
\dfrac{\psi'_{i}(\xi)}{\psi_{i}(b) -\psi_{i}(a)}] $
\end{proof}

%%%%%%%%%%%%%%%%%%%%%%%%%%%%%%%%%%%%%%%%%%%%%%%%%%%%%%%%%%%%%%%%%%%%%%%%%%%%%%%%%%%%%%%%%%%%%%%%
\section{Extended Generalized Flett's Mean Value Theorem}
If we use Flett's Theorem in Extended Generalized Mean Value Theorem
then what would the new theorem look like. A sincere effort is done
to develop this theorem. This theorem is named as Extended
Generalised Flett's Mean Value Theorem (EGMVT).

\begin{thm}
Let $ n \in \mathbb{N} $ and let $ f, g_{1}, g_{2}, g_{3},......,
g_{n} \in C [a,b]  \cap  D(a,b) $ also $ f, g_{1}, g_{2}, g_{3},
......, g_{n} $ are increasing functions and $ f^{n+1}(a)=f^{n+1}(b)
= 0 $ then $ \exists $ a $ \xi \in (a,b) $ s.t.
\begin{equation}
f^{n+1}(\xi)={_a^b}K(f)[\sum_{i=1}^{n}\dfrac{g'_{i}(\xi)}{n({_a^b}K(g_{i})}]
\end{equation}
\end{thm}

\begin{proof}
Without loss of generality assume $ f^{n+1}(a)=f^{n+1}(b)=0 $. Define a function \\
\[
G(x)=
\begin{cases}
 \hfill n(x-a){_a^x}K(f^{n})- \sum_{i=1}^{n} \dfrac{{_a^b}K(f)}{{_a^b}K(g_{i})}(g_{i}(x)-g_{i}(a))     \hfill , & \forall x\in [a,b]\\
\hfill \frac{1}{n+1}f^{n+1}(a)\hfill, & x=a\\
\end{cases}
\]
Now $ G \in C [a,b] \cap D^{n+1}(a,b) $ \\
\begin{equation}
G(x)=n(x-a)( \dfrac{f^{n}(x)-f^{n}(a)}{x-a})- \sum_{i=1}^{n}
\dfrac{f(b)-f(a)}{g_{i}(b)-g_{i}(a)}(g_{i}(x)-g_{i}(a))
\end{equation}

\noindent \begin{equation} G'(x)=nf^{n+1}(x) - (f(b)-f(a))
\sum_{i=1}^{n} \dfrac{g'_{i}(x)}{g_{i}(b)-g_{i}(a)}, x \in [a,b]
\end{equation}
So we can say $ \exists $  $ \xi \in (a,b) $ s.t. $ G'(\xi)=0 $ \\
From definition of $ G $ we can say $ G(a)=0 $.\\
If $ G(b)=0 $ then Rolle's Theorem guarantees existence of a point $
\xi \in (a,b) $ s.t. $ G'(\xi)=0 $. \\
Let $ G(b) \neq 0 $ and $ G(b)>0 $ (or $ G(b)<0 $). Then
\begin{eqnarray*}
% \nonumber to remove numbering (before each equation)
\noindent G'(b) & = & nf^{n+1}(b) - (f(b)-f(a)) \sum_{i=1}^{n} \dfrac{g'_{i}(b)}{g_{i}(b)-g_{i}(a)} \\
& = & - (f(b)-f(a)) \sum_{i=1}^{n} \dfrac{g'_{i}(b)}{g_{i}(b)-g_{i}(a)} ~~~~(since  f^{n+1}(b)=0 ) \\
& < &  0
\end{eqnarray*}

Now $ G\in C [a,b] $ and $ G'(b) < 0 $ \\
$ \Rightarrow G $ is strictly decreasing in $ b $ \\
$ \Rightarrow $ $ \exists $  $ x_{1} \in (a,b) $ s.t.
  $ G(x_{1})>G(b) $ \\
From continuity of $ G $ on $ [a,x_{1}] $ and from relation $ 0 =
G(a) < G(b) < G(x_{1}) $.Therefore we deduce from Darboux
Intermediate Value Theorem,

\noindent $ \exists $ $ x_{2} \in (a,x_{1}) $ s.t. $ G(x_{2})=G(b) $ \\
since $ G \in C [x_{2},b] \cap D^{n+1}(x_{2},b) $ \\
therefore, from Rolle's Theorem we have $ G'(\xi) = 0 $ for some $
\xi \in (x_{2},b) \subset (a,b) $

\begin{equation}
\noindent \Rightarrow  nf^{n+1}(\xi) - (f(b)-f(a)) [\sum_{i=1}^{n}
\dfrac{g'_{i}(\xi)}{g_{i}(b)-g_{i}(a)}] = 0
\end{equation}
\noindent \begin{equation} \Rightarrow  nf^{n+1}(\xi) = (f(b)-f(a))
[\sum_{i=1}^{n} \dfrac{g'_{i}(\xi)}{g_{i}(b)-g_{i}(a)}]
\end{equation}

\begin{eqnarray*}
\Rightarrow  f^{n+1}(\xi) & = & \frac{f(b) - f(a)}{n}[\sum_{i=1}^{n}\dfrac{g'_{i}(\xi)}{g_{i}(b)-g_{i}(a)}] \\
& = & {_a^b}K(f)[\sum_{i=1}^{n}
\dfrac{g'_{i}(\xi)}{n({_a^b}K(g_{i}))}]
\end{eqnarray*}
\end{proof}

\begin{eg}
Consider $ f(x)= x^{4}+1 , g_{1}(x)= x^2 -2x+12 , g_{2}(x)= x^{2}+4x
$ on an interval $ [0, 4] $
\end{eg}
\noindent \textbf{Solution}- We are given 3 functions $
f,g_{1},g_{2} $. We will check the conditions of \textbf{EGFMVT} one
by one as follows :
\begin{enumerate}
\item All $ f,g_{1},g_{2} $ are continuous on $ [0,4] $
\item All  functions are differentiable on $ (0,4) $ with $ g'_{i}(x) \neq 0 , \forall x \in (0,4) $ for $ i=1,2 $
\item Also $ f,g_{1},g_{2} $ are increasing functions at the endpoint $ 4 $.\\
All the conditions of theorem are satisfied, hence $ \exists $  $
\xi \in (0,4) $ s.t.
\end{enumerate}
\begin{equation}
f^{(3)}(\xi) = {_0^4}K(f)[\sum_{i=1}^{2} \dfrac{g'_{i}(\xi)}{2
{_0^4}K(g_{i})}]
\end{equation}
$ \Rightarrow 4 \xi = 4 \\
  \Rightarrow   \xi = 1 \\
  \Rightarrow   \xi \in (0,4) $


\normalsize
\begin{thebibliography}{}

\bibitem{Flett58}T. M. Flett: A mean value theorem. Math. Gazette 42 (1958),38–39.
\bibitem{Flett12} Ondrej Hutnik and Jana Molnarova, On Flett's mean value theorem, Aequationes Mathematicae VOl.89, pp. 1133–1165, 2015.
\bibitem{Phi14} Phillip Mafuta, Extended Generalized Mean Value Theorem for Functions of One Variable. IOSR Journal of Mathematics (IOSR-JM), Volume 10, Issue 2 Ver. II (Mar-Apr. 2014), PP 39-40.

\end{thebibliography}
\end{document}